\documentclass[12pt]{extarticle}
\usepackage[utf8]{inputenc}
\usepackage[english]{babel}
\usepackage{indentfirst}
\usepackage{amsmath}
\usepackage{amsfonts}
\usepackage{amssymb}
\usepackage{amsthm}
\usepackage{textcomp}
\usepackage{hyperref}
\usepackage[english]{varioref}
\usepackage[left=2.5cm,right=2.5cm,
top=2.5cm,bottom=2.5cm,bindingoffset=0cm]{geometry}
\usepackage{caption}
\usepackage{multirow}
\usepackage{mathabx}

\newtheorem{problem}{Problem}
\newtheorem{theorem}{Theorem}

\newtheorem{corollary}{Corollary}
\newtheorem{remark}{Remark}

\renewcommand{\P}{\mathbb{P}}
\renewcommand{\O}{\mathcal{O}}
\renewcommand{\mod}[3]{#1 \equiv #2 \ (\mathrm{mod} \ #3)}

\renewcommand{\t}[2]{#1\!\times\!#2}

\newcommand{\simTo}{\xrightarrow{\smash{\raisebox{-0.7ex}
			{\ensuremath{\scriptstyle\!\!\sim}}}}}

\newcommand{\cd}{\!\cdot\!}
\newcommand{\Fr}{\mathrm{Fr}}

\newcommand{\F}[1]{\mathbb{F}_{\!#1}}

\newcommand{\A}[1]{\mathbb{A}^{\!#1}}
\newcommand{\rmA}{\mathrm{A}}
\newcommand{\rmD}{\mathrm{D}}
\newcommand{\rmE}{\mathrm{E}}
\newcommand{\N}{\mathbb{N}}

\newcommand{\Z}{\mathbb{Z}}

\newcommand{\E}{\mathcal{E}}
\newcommand{\K}{\mathcal{K}}

\newcommand{\I}{\mathrm{I}}
\newcommand{\II}{\mathrm{II}}
\newcommand{\IV}{\mathrm{IV}}

\newcommand{\inv}[1]{#1^{-\!1}}

\newcommand{\MW}{\mathrm{MW}}

\newcommand{\rk}{\mathrm{rk}}

\newcommand{\Leg}[2]{\big(\frac{#1}{#2}\big)}

\let\oldbibliography\thebibliography
\renewcommand{\thebibliography}[1]{
	\oldbibliography{#1}
	\setlength{\itemsep}{-0.5pt}
}

\begin{document}
	
\setlength{\medmuskip}{1\medmuskip}
\setlength{\thickmuskip}{1\thickmuskip}	

\begin{center}
	
	{\large\bf Hashing to elliptic curves of $j=0$ and Mordell--Weil groups}
	
	% Хеширование в эллиптические кривые $j$-инварианта $0$ и группы Морделла--Вейля
	
	\vspace{0.5cm}
	
	{\large Koshelev Dmitrii} \footnote{
		Web page: https://www.researchgate.net/profile/Dimitri\_Koshelev\\ 
		\phantom{\hspace{0.5cm}} Email: dishport@ya.ru\\
		This work was supported by the grant RFBR № 19-31-90029$\backslash$19
	}
	
	\vspace{0.25cm}
	
	Versailles Laboratory of Mathematics, Versailles Saint-Quentin-en-Yvelines University
	Center for Research and Advanced Development, Infotecs \\
	Algebra and Number Theory Laboratory, Institute for Information Transmission Problems \\
	
	\vspace{0.25cm}		
	
\end{center}

{\bf \abstractname.} Consider an ordinary elliptic curve $E_b\!: y^2 = x^3 - b$ (of $j$-invariant $0$) over a finite field $\F{q}$ such that $\sqrt[3]{b} \notin \F{q}$. This article tries to resolve the problem of constructing a rational $\F{q}$-curve on the Kummer surface of the direct product $\t{E_b}{E_b^\prime}$, where $E_b^\prime$ is the quadratic $\F{q}$-twist of $E_b$. More precisely, we propose to search such a curve among infinite order $\F{q}$-sections of some elliptic surface of $j=0$, analysing its Mordell--Weil group. Unfortunately, we prove that it is just isomorphic to $\Z/3$.

% Consider an ordinary elliptic curve $E_b\!: y^2 = x^3 - b$ (of $j$-invariant $0$) over a finite field $\mathbb{F}_{\!q}$ such that $\sqrt[3]{b} \notin \mathbb{F}_{\!q}$. This article tries to resolve the problem of constructing a rational $\mathbb{F}_{\!q}$-curve on the Kummer surface of the direct product $E_b \!\times\! E_b^\prime$, where $E_b^\prime$ is the quadratic $\mathbb{F}_{\!q}$-twist of $E_b$. More precisely, we propose to search such a curve among infinite order $\mathbb{F}_{\!q}$-sections of some elliptic surface of $j=0$, analyzing its Mordell--Weil group. Unfortunately, we prove that it is just isomorphic to $\mathbb{Z}/3$.

% Рассмотрим обыкновенную эллиптическую кривую $E_b\!: y^2 = x^3 - b$ (с $j$-инвариантом $0$) над конечным полем $\mathbb{F}_{\!q}$ таком, что $\sqrt[3]{b} \notin \mathbb{F}_{\!q}$. Данная статья пытается разрешить проблему построения рациональной $\mathbb{F}_{\!q}$-кривой на куммеровой поверхности прямого произведения $E_b \!\times\! E_b^\prime$, где $E_b^\prime$ это квадратичный $\mathbb{F}_{\!q}$-твист кривой $E_b$. Говоря точнее, мы предлагаем искать такую кривую среди $\mathbb{F}_{\!q}$-сечений бесконечного порядка на некоторой эллиптической поверхности с $j=0$, анализируя ее группу Морделла--Вейля. К сожалению, мы доказываем, что она просто-напросто изоморфна группе $\mathbb{Z}/3$.

\vspace{0.25cm}


\section*{Introduction}
\addcontentsline{toc}{section}{Introduction}

Many protocols of {\it pairing-based cryptography} \cite{ElMrabetJoye} use some (not necessarily injective or surjective) map \mbox{$h\!: \F{q} \to E_b(\F{q})$} (often called {\it hashing}) from a finite field $\F{q}$ (of characteristic $p > 3$) to an ordinary elliptic curve $E_b\!: y_0^2 = x_0^3 - b$, whose $j$-invariant is $0$. A review of this topic is represented, for example, in \mbox{\cite[Chapter 8]{ElMrabetJoye}}. At the moment, most really used curves $E_b$ (so-called {\it pairing-friendly curves} \mbox{\cite[\S 4]{ElMrabetJoye}}) have the restriction $\sqrt[3]{b} \notin \F{q}$, that is $2 \nmid |E_b(\F{q})|$.

In order to construct the hashing $h$ one can try to use the {\it simplified SWU} ({\it Shallue--van de Woestijne--Ulas}) {\it method}, which we explain in the introduction of \cite{Koshelev}. This method raises a very interesting algebraic geometry task (cf. \mbox{\cite[Problem 1]{Koshelev}}) of finding a rational (possibly singular) $\F{q}$-curve (and its proper $\F{q}$-parametrization) on the {\it Kummer surface} $\K_2$ (see, e.g., \cite[\S 2]{Koshelev}) of the direct product $\t{E_b}{E_b^\prime}$, where $E_b^\prime$ is the quadratic $\F{q}$-twist of $E_b$ (see, e.g., \cite[\S 2.3.6]{ElMrabetJoye}). Unfortunately, the severe requirement $\sqrt[3]{b} \notin \F{q}$ makes the task very difficult. 

Since $E_b$ is assumed to be ordinary, we have $\mod{q}{1}{3}$, that is $\omega := (-1 + \sqrt{-3})/2$ (in other words, $\omega^3 = 1$, $\omega \neq 1$) lies in $\F{q}$. Also, take any element $c \in (\F{q}^*)^3$ such that $c \notin (\F{q}^*)^2$. By the second assumption on $c$ we get the equations
$$
E_b^\prime \!: c y_1^2 = x_1^3 - b, \qquad\qquad
\K_2\!: (x_1^3 - b)t^2 = c(x_0^3 - b) \quad \subset \quad \A{3}_{(x_0,x_1,t)},
$$
where $t := y_0/y_1$. 

There is on $\K_2$ the natural elliptic fibration $(x_0,x_1,t) \mapsto t$ (so-called {\it Inose fibration} \cite{Inose}), however we do not know any its $\F{q}$-section. Instead, we apply the base change $t \mapsto t^3$ and obtain the elliptic surface
$$
\K_6\!: (x_1^3 - b)t^6 = c(x_0^3 - b) \quad \subset \quad \A{3}_{(x_0,x_1,t)}
$$
having the section $\O := (t^2: \sqrt[3]{c}: 0) \in \P^2_{(X_0:X_1:X_2)}$, where $x_0 = X_0/X_2$, $x_1 = X_1/X_2$.

The surface $\K_6$ is sometimes called {\it Kuwata surface} \cite{Kuwata} (also see \cite{Kloosterman}, \cite{KumarKuwata}, \cite{Shioda2000}, \cite{Shioda2008}). It is worth noting that $\K_2$, $\K_6$ are $K3$ surfaces (see, e.g., \mbox{\cite[\S 12]{ShiodaSchutt}}). Moreover, in accordance with \mbox{\cite[Theorem 8.1]{Shioda2008}} they are {\it singular} \cite{Inose}, \mbox{\cite[\S 13]{ShiodaSchutt}}, \cite{ShiodaInose}, that is their Picard $\overline{\F{q}}$-numbers are equal to $20$ (the highest possible one for ordinary $K3$ surfaces). By the way, the Picard $\F{q}$-number of $\K_2$ equals $8$ (see, e.g., \cite[\S 2]{Koshelev}), which is the smallest possible one for Kummer surfaces of the direct product of any two elliptic $\F{q}$-curves.

Because of Lüroth's theorem any rational $\F{q}$-curve on $\K_6$ gives (by means of $t \mapsto t^3$) the rational one on $\K_2$. In this regard, it is natural to study the {\it Mordell--Weil group} $\MW(\K_6)$ (see \mbox{\cite[\S 6]{ShiodaSchutt}}) and explicitly derive one of its non-zero elements (whose canonical height \mbox{\cite[\S 11.6]{ShiodaSchutt}} is as low as possible). This approach is deployed in \cite[\S 1]{KuwataWang}, \cite[\S 1]{Mestre}, where any elliptic curve of $j \neq 0, 1728$ is taken instead of $E_b$. 

We will denote for clarity by $\MW\big( \widebar{\K_6} \big)$ the Mordell--Weil group of all sections of $\K_6$ (not necessarily defined over $\F{q}$). According to \mbox{\cite[\S 1]{Shioda2000}}, \mbox{\cite[Lemma 6.2]{Shioda2008}} we have $\MW\big( \widebar{\K_6} \big) \simeq \Z^6 \oplus \Z/3$. In particular, the torsion subgroup is generated by any of two sections $(t^2: \omega^j \sqrt[3]{c}: 0)$, where $j \in \{1,2\}$. By contrast, we prove that $\MW(\K_6) \simeq \Z/3$, that is the Mordell--Weil $\F{q}$-rank of $\K_6$ is equal to $0$. Unfortunately, since the torsion sections lie at infinity (i.e., on the line $X_2 = 0$), we can not use them to construct the hashing $h$ by the simplified SWU method.

{\bf Acknowledgements.} The author expresses his deep gratitude to his scientific advisor M. Tsfasman and thanks A. Trepalin for the help and useful comments.

\section{Main result}
\addcontentsline{toc}{section}{Main result}

The theory of elliptic surfaces over $\P^1$ (i.e., elliptic curves over the function field in one variable $t$) is well represented, for example, in \cite{ShiodaSchutt} (and in \cite{Ulmer} for the case of a finite field $\F{q}$). By abuse of notation, we will denote an elliptic $\F{q}$-surface $S$ and its generic $\F{q}(t)$-fiber by the same letter. Besides, let us identify $S$ with its (unique) Kodaira--Néron model.

\begin{theorem}
	The Mordell--Weil group $\MW(\K_6)$ is isomorphic to $\Z/3$.
\end{theorem}

\begin{proof}
	First of all, we transform $\K_6$ (with $\O$ as the zero section) to its globally minimal \cite[\S 8.2]{ShiodaSchutt} Weierstrass form
	$$
	\E\!: y^2 = x^3 + \left(\! \dfrac{t^6-c}{2b^2} \!\right)^{\!\!2}
	$$
	by means of the $\F{q}(t)$-isomorphism
	$$
	\begin{array}{ll}
	\varphi\!: \K_6 \simTo \E,\qquad & \inv{\varphi}\!: \E \simTo \K_6, \\[\medskipamount]
	\varphi = \begin{cases}
	x := \dfrac{t^6-c}{ b (x_0\sqrt[3]{c} - t^2x_1) },\\[\bigskipamount]
	y := \dfrac{\sqrt{-3} (x_0\sqrt[3]{c} + t^2x_1)}{-2b} \cdot x,
	\end{cases} \qquad & 
	\inv{\varphi} = \begin{cases}
	x_0 := \dfrac{ 2b^2y - \sqrt{-3}(t^6-c) }{ -2\sqrt{-3}\cdot b\sqrt[3]{c} \cdot x },\\[\bigskipamount]
	x_1 := \dfrac{ 2b^2y + \sqrt{-3}(t^6-c) }{ -2\sqrt{-3}\cdot bt^2 \cdot x }.
	\end{cases}
	\end{array}
	$$
	These formulas are verified in \cite{Magma}.
	
	Also, for $j \in \Z/6$ consider the elliptic surfaces given by globally minimal Weierstrass forms
	$$
	\E_j\!: y^2 = x^3 + t^j\! \left(\! \dfrac{t-c}{2b^2} \!\right)^{\!\!2}\!\!.
	$$
	Note that $\E_j$ is a Weierstrass form for $\E_{4-j}$ minimal at $t = \infty$. By \cite[\S 4.10]{ShiodaSchutt} the surfaces $\E_0, \E_1, \E_2$ are geometrically rational, but $\E_5$ is, in turn, $K3$ one. Besides, according to \cite[Lemma 2.1]{Top} we have the following identities between Mordell--Weil ranks:
	$$
	\rk(\E) = \sum_{j=0}^5 \rk(\E_j), \qquad\qquad
	\rk\big( \widebar{\E} \big) = \sum_{j=0}^5 \rk\big( \widebar{\E_j} \big).
	$$
	
	Let $\rho\big( \widebar{\E_j} \big)$ be the Picard $\overline{\F{q}}$-number of $\E_j$. Using Tate's algorithm \cite[\S 4.2]{ShiodaSchutt} (also see \cite{Magma}), the main theorem of \cite{OguisoShioda}, and the Shioda--Tate formula \mbox{\cite[\S 3.5]{Ulmer}}, we immediately obtain all cells of Table \ref{surfacesEj} except for $\MW\big( \widebar{\E_5} \big)$, $\rho\big( \widebar{\E_5} \big)$. Since $\rk\big( \widebar{\E} \big) = 6$, we also get $\rk\big( \widebar{\E_5} \big) = 0$ and, as a result, $\E_5$ is a singular $K3$ surface. Moreover, $\E_5$ is so-called {\it extremal} elliptic surface, hence by the row 297 of \cite[Table 2]{ShimadaZhang} we have $\MW\big( \widebar{\E_5} \big)_{tor} = 0$. Thus Table \ref{surfacesEj} is completely filled.
	
	It remains to prove that $\rk(\E_1) = \rk(\E_2) = 0$. For $k \in \Z/3$ consider the sections
	$$
	P_k := \left( \dfrac{\omega^k (t-c)}{-b\sqrt[3]{4b}},\ \dfrac{\sqrt{c}(t-c)}{2b^2} \right)\!, \qquad\qquad
	Q_k := \left( \dfrac{\omega^k \sqrt[3]{c} \cdot t}{b\sqrt[3]{b}},\ \dfrac{t^2+c t}{2b^2} \right)
	$$
	of $\E_1$ and $\E_2$ respectively. Since each triple have the same $y$-coordinate, we obviously get
	$$
	P_0 + P_1 + P_2 = \O,\qquad\qquad Q_0 + Q_1 + Q_2 = \O.
	$$ 
	
	The canonical height matrices
	$$
	\widehat{h}_{L_1} = \left(\begin{matrix}
	1/3	& -1/6 \\
	-1/6 & 1/3 
	\end{matrix}\right), \qquad\qquad
	\widehat{h}_{L_2} = \left(\begin{matrix}
	2/3	& -1/3 \\
	-1/3 & 2/3 
	\end{matrix}\right)
	$$
	on the lattices $L_1 := \langle P_0, P_1 \rangle$ and $L_2 := \langle Q_0, Q_1 \rangle$ are not hard to derive, looking at \cite[Theorem 8.6]{Shioda1990}. Instead, we use in \cite{Magma} one of Magma functions in order to reduce the amount of computations. The given matrices are non-degenerate, hence the sections $P_0, P_1$ (resp. $Q_0, Q_1$) are particularly linearly independent. 
	
	Besides, we have the following possible Frobenius actions $\Fr$ on $L_1$:
	$$
	\left(\begin{matrix}
	-1	&	0 \\
	0	&	-1 
	\end{matrix}\right), \qquad\qquad
	\left(\begin{matrix}
	1	&	-1 \\
	1	&	0 
	\end{matrix}\right), \qquad\qquad
	\left(\begin{matrix}
	0	&	1 \\
	-1	&	1 
	\end{matrix}\right)	
	$$
	if the third power residue symbol $\Leg{4b}{q}_{\!3} = 1, \omega, \omega^2$ respectively. In turn, the Frobenius $\Fr$ on $L_2$ is given by one of matrices
	$$	
	\left(\begin{matrix}
	1	&	0 \\
	0	&	1 
	\end{matrix}\right), \qquad\qquad
	\left(\begin{matrix}
	-1	&	1 \\
	-1	&	0 
	\end{matrix}\right), \qquad\qquad
	\left(\begin{matrix}
	0	&	-1 \\
	1	&	-1 
	\end{matrix}\right)
	$$
	if $\Leg{b}{q}_{\!3} = 1, \omega, \omega^2$ respectively. However the first case is ruled out by our assumption. In all remaining cases,
	$$
	\rk(\E_1) = \rk\big( L_1^\Fr \big) = 0,\qquad\qquad \rk(\E_2) = \rk\big( L_2^\Fr \big) = 0,
	$$
	where $L_1^\Fr\!$, $L_2^\Fr$ are the sublattices of $\Fr$-invariants. Thus the theorem is proved.
\end{proof}

\begin{table}[h]
	\centering
	\begin{tabular}{c|c|c|c|c}
		$\E_j$ & singular fibers & lattice $\mathrm{T}$ & $\MW\big( \widebar{\E_j} \big)$ & $\rho\big( \widebar{\E_j} \big)$ \\ \hline\hline 
		
		$\E_0$ & $\IV$, $\IV^*$ & $\rmA_2 \oplus \rmE_6$ & $\Z/3$ & \multirow{3}{*}{$10$} \\ \cline{1-4}
		$\E_1$ & $\II$, $\IV$, $\I_0^*$ & $\rmA_2 \oplus \rmD_4$ & $\Z^2$ & \\ \cline{1-4}
		$\E_2$ & $3 \cd \IV$ & $\rmA_2^{\oplus \, 3}$ & $\Z^2\oplus \Z/3$ & \\ \hline 
		$\E_5$ & $\IV$, $2 \cd \II^*$ & $\rmA_2 \oplus \rmE_8^{\, \oplus \, 2}$ & $0$ & $20$
		
	\end{tabular}
	\caption{The surfaces $\E_j$, where $\E_0 \simeq_{\F{q}(t)} \E_4$, $\E_1 \simeq_{\F{q}(t)} \E_3$}
	\label{surfacesEj}
\end{table}

\begin{remark}
	As we see in the proof, if $\Leg{b}{q}_{\!3} = 1$, then $L_2^\Fr = L_2$. At the same time, since $\rk\big( L_1^\Fr \big) = 0$ as before, we have $\rk(\K_6) = 2$.
\end{remark}

Since the base change $t \mapsto t^3$ obviously transforms an affine (i.e., $X_2 \neq  0$) $\F{q}$-section of $\K_2$ to that of $\K_6$, we also establish

\begin{corollary}
	The elliptic $\F{q}$-surface $\K_2$ is not Jacobian, that is it has no an $\F{q}$-section.
\end{corollary}

\section{Further questions}
\addcontentsline{toc}{section}{Further questions}

Let us shortly discuss what other Jacobian elliptic $\F{q}$-fibrations can be potentially used to construct a rational $\F{q}$-curve on the Kummer surface $\K_2$. First, it is very natural to formulate

\begin{problem}
	Is there a number $n \in \N$ such that the Mordell--Weil group $\MW(\K_{6n})$ of the elliptic surface
	$$
	\K_{6n}\!: (x_1^3 - b)t^{6n} = c(x_0^3 - b) \quad \subset \quad \A{3}_{(x_0,x_1,t)}
	$$
	$($with $(t^{2n}: \sqrt[3]{c}: 0)$ as the zero section$)$ is of non-zero rank?
\end{problem}

The base change $t \mapsto t^{3n}$ allows to transfer rational $\F{q}$-curves on $\K_{6n}$ to rational ones on $\K_2$. As well as for $\K_6$ we verify in \cite{Magma} that 
$$
y^2 = x^3 + \left(\! \dfrac{t^{6n}-c}{2b^2} \!\right)^{\!\!2}
$$
is a globally minimal Weierstrass form for $\K_{6n}$. Therefore the arithmetic genus of its Kodaira--Néron model is equal to $2n$ and hence for $n>1$ we deal with a surface of Kodaira dimension one.

We also can consider elliptic fibrations immediately on $\K_2$. All of them are classified (without explicit formulas) in \mbox{\cite[Table 1.3]{Nishiyama}} over an algebraically closed field. This is a good question which ones are Jacobian $\F{q}$-fibrations. Finally, one may wonder about the existence of another dominant rational maps from (elliptic) $K3$ surfaces onto $\K_2$. This topic is highlighted, for example, in \mbox{\cite[\S 3]{Ma}}.

\end{document}